\documentclass[11pt]{amsart}

\usepackage{amsmath,amssymb,color,graphicx,fullpage}
\usepackage{hyperref}
\usepackage{tikz-cd}
\usepackage{multirow}
\usepackage{pdflscape}
\usepackage{longtable}
\usepackage{supertabular}
\usepackage{tabularx}
\usepackage{array}
\usepackage{xcolor}
\usepackage{float}
\usepackage{multicol}

\hypersetup{colorlinks=true,urlcolor=blue,citecolor=blue,linkcolor=blue}

\theoremstyle{plain}

\newtheorem{corollary}[equation]{Corollary}
\newtheorem{theorem}[equation]{Theorem}

\theoremstyle{remark}
\newtheorem{example}[equation]{Example}
\newtheorem{remark}[equation]{Remark}

\theoremstyle{definition}

\newcommand{\A}{\mathbb A}

\newcommand{\F}{\mathbb F}

\newcommand{\HH}{\mathbb H}

\newcommand{\Q}{\mathbb Q}
\newcommand{\R}{\mathbb R}

\newcommand{\Z}{\mathbb Z}

\newcommand{\fraka}{\mathfrak a}
\newcommand{\frakb}{\mathfrak b}

\newcommand{\la}{\langle}
\newcommand{\ra}{\rangle}

\newcommand{\Qbar}{\overline{\Q}}

\newcommand{\w}{\omega}

\DeclareMathOperator{\End}{End}
\DeclareMathOperator{\Gal}{Gal}
\DeclareMathOperator{\Jac}{Jac}

\DeclareMathOperator{\PSL}{PSL}

\DeclareMathOperator{\Nm}{Nm}

\DeclareMathOperator{\Aut}{Aut}

\DeclareMathOperator{\CM}{CM}

\newcount\n
\def\tablebody{}
\makeatletter
\loop\ifnum\n<100
        \advance\n by1
        \protected@edef\tablebody{\tablebody
                \textbf{\number\n.}& shortText
                \tabularnewline
        }
\repeat

\makeatletter
\let\mcnewpage=\newpage
\newcommand{\TrickSupertabularIntoMulticols}{%
  \renewcommand\newpage{%
    \if@firstcolumn
      \hrule width\linewidth height0pt
      \columnbreak
    \else
      \mcnewpage
    \fi
  }%
}
\makeatother

\begin{document}

\title{Rational points on Atkin-Lehner quotients of geometrically hyperelliptic Shimura curves}

\author{\sc Oana Padurariu}
\address{Oana Padurariu \\
Dept. of Mathematics \& Statistics\\  
Boston University\\
111 Cummington Mall\\
Boston, MA 02215\\
USA}
\urladdr{https://sites.google.com/view/oana-padurariu/home}
\email{oana@bu.edu}

\author{\sc Ciaran Schembri}
\address{Ciaran Schembri \\
Dept. of Mathematics\\  
Dartmouth College\\
 27 N Main St\\
Hanover, NH 03755\\
USA}
\urladdr{https://math.dartmouth.edu/~cschembri/}
\email{ciaran.schembri@dartmouth.edu}

\maketitle

\begin{abstract}
Guo and Yang give defining equations for all geometrically hyperelliptic Shimura curves $X_0(D,N)$.  In this paper we compute the $\Q$-rational points on the Atkin-Lehner quotients of these curves using a variety of techniques.  We also determine which rational points are CM for many of these curves.
\end{abstract}

\bibliographystyle{alpha}


\section{Introduction}

It is an important problem to study the modular curves $X_0(N)$ and their rational points. These curves are the coarse moduli spaces for elliptic curves with a $\Gamma_0(N)$-level structure and an understanding of their rational points leads to a classification of elliptic curves equipped with an isogeny (cf. \cite{Mazur78}).

There is a straightforward generalization of these modular curves: one may think of $\Gamma_0(N)$ as a discrete subgroup of the units of the rational indefinite quaternion algebra $M_2(\Q)$ (see \S\ref{Background} for definitions).  So let $D$ be the discriminant of a rational indefinite quaternion algebra.  There is an analogue of $\Gamma_0(N)$ for the quaternion algebra of discriminant $D$ denoted $\Gamma_0(D,N)$.  This group acts as isometries on the upper half plane and the quotient of this action is the \textit{Shimura curve} $X_0(D,N)$ (if $D>1$ and a modular curve if $D=1$). 

Shimura curves are also coarse moduli spaces, in the case of $X_0(D,N)$ it is for principally polarized abelian surfaces equipped with an embedding of an Eichler order of level $N$ into the endomorphism ring.  Thus a study of the points on $X_0(D,N)$ sheds light on the classification of these surfaces with many endomorphisms.

Whilst modular and Shimura curves share many properties there are some important distinctions.  Firstly,  the Shimura curves $X_0(D,N)$ have no real-valued points \cite{Shimura75}.  One may profitably study points over totally complex fields (cf. \cite{Jordan86, SY04}). However, if we are concerned with the abelian surfaces being defined over $\Q$ then we must consider rational points on \textit{Atkin-Lehner} quotients of $X_0(D,N)$.

Secondly,  it is quite difficult to compute explicit defining equations for Shimura curves, which means that there is a relative lack of defining equations in the literature. The main reason for this is that Shimura curves have no cusps and thus there are no Fourier expansions to work with.  Recently,  Guo and Yang computed defining equations for all Shimura curves $X_0(D,N)$ which are geometrically hyperelliptic \cite{GY17}.  The main aim of our paper is to compute the $\Q$-rational points on the Atkin-Lehner quotients of these curves.

\begin{theorem}
Let $X_0(D,N)$ be a Shimura curve which is hyperelliptic over $\Qbar$ and $W$ a subgroup of Atkin-Lehner involutions. Then defining equations for the Atkin-Lehner quotient curve $X_0(D,N)/W$ have been computed and in the case that the quotient curve has finitely many rational points the set $(X_0(D,N)/W)(\Q)$ is given explicitly (see \S\ref{Appendix} for details).  Furthermore,  when the level $N$ is $1$ and the quotient curve has finitely many rational points it is known which of these points are CM.  
\end{theorem}

In total there are $44$ geometrically hyperelliptic Shimura curves $X_0(D,N)$ and $380$ Atkin-Lehner quotients.  The curves and their rational points are organized into a database which can be accessed online\footnote{\url{https://github.com/ciaran-schembri}}. In the appendix we include information about the data that has been computed, including the genus and the number of rational points for each curve; and for curves of level one we display the rational points and whether they are CM as appropriate.

The remainder of the paper is organized as follows: in \S\ref{ShimuraCurves} we give a background on Shimura curves and their quotients,  including a brief survey on defining equations in the literature; in \S\ref{Computations} we detail how we carried out the computations and in \S\ref{Examples} we give some examples that might be of interest including some violations of the Hasse principle.

\subsection{Acknowledgements} We thank John Voight for many useful conversations and help with computing CM points. We also thank Jacob Swenberg for helping us identify issues with certain CM orders. We would also like to thank Jennifer Balakrishnan, Pete L. Clark, Stevan Gajovi\'c, Steffen M\"uller and Frederick Saia for very helpful discussions. The authors would also like to thank Sam Schiavone and Juanita Duque-Rosero for helping us organize our data into the LMFDB. The first named author was supported by NSF grant DMS-1945452 and  Simons Foundation grant \#550023 and the second named author was supported by a Simons Collaboration grant \#550029. We thank the referee for carefully reading our paper and suggesting many valuable improvements.

 \section{Shimura curves}\label{ShimuraCurves}
 
 \subsection{Background}\label{Background}
 For a background on quaternion algebras we refer the reader to \cite{VoightQA}.
 A quaternion algebra $B$ defined over $\Q$ is a central simple $\Q$-algebra of dimension $4$.  Equivalently,  $B$ is a quaternion algebra defined over $\Q$ if and only if there are elements $i,j \in B$ such that
 $$B \simeq \Q \cdot 1+ \Q \cdot i + \Q \cdot j + \Q \cdot ij$$
 and $i^2,j^2 \in \Q^\times$ with $ij = -ji.$
 Note that $M_2(\Q)$ is an example of a quaternion algebra.  An order of $B$ is a full rank $\Z$-lattice which is also a subring.
 
 For any place $v$ of $\Q$,  we can consider the completion $B\otimes_{\Q} {\Q}_{v}$ at $v$ and $B$ is said to ramified at $v$ if the completion at $v$ is a division algebra.  The set of ramified places of $B$ is finite,  of even cardinality and uniquely determines the isomorphism class of a quaternion algebra.  The discriminant $D$ of $B$ is the product of the finite ramified places of $B$.  We say that $B$ is indefinite if $B$ is unramified at the archimedean place.

Throughout $B$ will be an indefinite quaternion algebra defined over the rationals of discriminant $D$ and $O$ will denote a maximal order.  There are embeddings 
\begin{align*}
\iota_\infty &: B \hookrightarrow M_2(\R) \simeq B \otimes_{\Q} \R \\
\iota_p &: O \hookrightarrow M_2(\Z_p) \simeq O \otimes_{\Z} \Z_p
\end{align*}
where $p$ is any integer prime not dividing $D$.  As with modular curves we have a notion of level: Eichler orders of level $N$ are defined to be orders of the form
$$O_N  := \{ \ x \in O \ | \ \iota_p(x) \equiv \begin{pmatrix} \ast & \ast  \\ 0 & \ast \end{pmatrix} \ \text{mod} \ p^e \ \text{for all} \ p^e ||N \ \}.$$

There is an action of the norm 1 elements of $O_N$ via $\iota_\infty (O_N^1) / \{ \pm 1 \} \leq \PSL_2(\R)$ on the upper half plane 
$$\HH = \{ \ x+iy \ | \ y>0 \ \}$$
by
$$ \gamma \cdot z := \iota_\infty(\gamma) \cdot z = \begin{pmatrix} a & b \\ c& d \end{pmatrix} \cdot z = \frac{az+b}{cz+d}.$$

Consider the Riemann surface
$$\iota_\infty (O_N^1) / \{ \pm 1 \} \backslash \HH.$$
If $B$ has discriminant 1 then this Riemann surface is the classical modular curve $X_0(N).$

From now on we shall assume that $B$ is not isomorphic to $M_2(\Q)$, or equivalently that $D>1.$  By work of Shimura (\cite[Main Theorem I]{Shimura67}, \cite[Theorem 2.5]{Shimura70}) there is an algebraic curve which we denote $X_0(D,N),$ such that there is an open immersion of Riemann surfaces $\iota_\infty (O_N^1) / \{ \pm 1 \} \backslash \HH \hookrightarrow X_0(D,N)$ which is a biregular isomorphism.  The curve $X_0(D,N)$ has a canonical model defined over $\Q$ \cite[\S3]{Shimura67},  it is called a Shimura curve and it has level $N$.  

As with modular curves,  Shimura curves arize naturally as a moduli problem.  The curve $X_0(D,N)$ parameterizes pairs $(A,\iota)$ where 
\begin{itemize}
\item $A$ is an abelian surface;
\item $\iota : O_N \hookrightarrow \End(A)$ is an embedding.
\end{itemize}
More precisely,  in the category of $\Q$-schemes, $S\longmapsto \{(A,\iota)_{/S}\}$ is coarsely represented by the curve $X_0(D,N)$ \cite{Deligne71}.

We call such a pair $(A,\iota)$ a QM-surface.  It is well known that Shimura curves have no real points (\cite[Theorem 0]{Shimura75}, \cite[\S3]{Ogg85}),  i.e.
$$X_0(D,1)(\R) = \emptyset.$$
A consequence of this fact is that there are no abelian surfaces defined over $\R$ such that there is an embedding $\iota : O_N \hookrightarrow \End(A)$ also defined over $\R$.
So in contrast with modular curves,  it is not possible to consider their rational points; rather,  we must study the rational points on quotients of these Shimura curves, which we now define.

Consider the normalizer group
$$N_{B_{>0}^\times}(O_N^1) := \{ \ \alpha \in B_{>0}^\times \ | \ \alpha^{-1} O_N^1 \alpha = O_N^1 \ \},$$
where $B_{>0}^\times$ are the units in $B$ of positive norm.  Elements of $N_{B_{>0}^\times}(O_N^1)$ naturally define an automorphism of the Shimura curve $X_0(D,N)$,  with the scalars $\Q^\times$ acting trivially.  It is straightforward to check that two elements $\alpha, \beta$ of $N_{B_{>0}^\times}(O_N^1)$ define the same automorphism if and only if $\alpha \beta^{-1} \in \Q^\times O_N^1.$ We arrive at the definition of the group of Atkin-Lehner involutions:
$$W(D,N):= N_{B_{>0}^\times}(O_N^1) / \Q^\times O_N^1.$$

There is an identification $W(D,N) \simeq (\Z/2\Z)^{\# \{ p  | DN \ :  \ p \text{ prime}\}}$ and we will write $W(D, N) = \{ \ \w_m \ | \ m>0,  \ m|ND, \  \gcd(m,ND/m) = 1 \ \}$ \cite[Ch.  28]{VoightQA}. Note that $\w_m \cdot \w_n = \w_{mn/\gcd(m,n)^2}$.

Now for any subgroup $W \leq W(D,N)$ of involutions we can consider the Atkin-Lehner quotient 
$$X_0(D,N)/W.$$
It is possible for these quotient curves to have rational points.

The group $W$ also acts on the set of embeddings $\iota : O_N \hookrightarrow \End(A)$ and the curve $X_0(D,N)/W$ parameterizes surfaces $A$ up to identifying embeddings under the action of $W$.  The quotient by the full group parameterizes surfaces and `forgets' the embedding. For a detailed discussion of the modular interpretation of the Atkin-Lehner group see \cite[\S3.1]{Rotger04mod}.

For a totally complex field $K$, Jordan showed that a $K$-rational point on $X_0(D,1)$ represents a QM-surface defined over $K$ if and only if $K$ splits the quaternion algebra $B$, i.e. $B \otimes_{\Q} K \simeq M_2(K)$ \cite[Theorem 1.1]{Jordan86}.  It is also not necessarily the case that a rational point on an Atkin-Lehner quotient gives rise to an abelian surface defined over the rationals (cf. \cite[\S4]{BFGR06}).

We conclude this subsection with a description of certain special points on Shimura curves and their quotients called CM points.  Let $K = \Q(\sqrt{-d})$ be an imaginary quadratic field which admits an embedding $q : K \hookrightarrow B.$  Let $R \subset K$ be an order.  Then $R$ is said to be optimally embedded in $O_N$ if $q(K) \cap O_N = q(R).$

For an optimally embedded order $R$ there is exactly one fixed point $z$ of $\iota_\infty (O_N^1) / \{ \pm 1 \} \backslash \HH$ under the action of $\iota_{\infty}( q(R)).$ We say that $z$ is a CM point for $R$.  The set 
$$\CM(R) := \{ \ z \in \iota_\infty (O_N^1) / \{ \pm 1 \} \backslash \HH \ | \ \iota_{\infty}( q(R)) \cdot z = z, \ q : R \hookrightarrow B \ \text{an optimal embedding} \ \}$$
is the set of CM points for $R$. We identify the CM points on $\iota_\infty (O_N^1) / \{ \pm 1 \} \backslash \HH$ with the corresponding points on $X_0(D,N)$.  

There is the natural quotient map $\pi_W : X_0(D,N) \longrightarrow X_0(D,N)/W$ which we shall denote $\pi_m$ if $W = \la \w_m \ra.$ We say a point $Q \in X_0(D,N)/W$ is a CM point if any of its preimages on $X_0(D,N)$ is CM. 

\begin{remark}
CM points on $X_0(D,N)$ correspond to abelian surfaces with extra endomorphisms by the CM field K: $\End(A) \otimes_\Z \Q \simeq M_2(K).$ In fact, an alternative definition of a CM point for $R$ on $X_0(D,N)$ is a point in the fibre over $[(A,\iota)] \in X_0(D,1)$ where $\End(A) \otimes_\Z \Q \simeq M_2(K)$ and the QM-equivariant endomorphisms are equal to $R$ (cf. \cite{Clark03}).
\end{remark}

A special case of the Coleman conjecture (\cite[p.1]{BFGR06}) predicts that there are only finitely possibilities for the geometric endomorphism ring of an abelian surface $A/\Q.$ An abelian surface $A/\Q$ such that $\End(A_\Q) \otimes_\Z \Q \simeq \Q(\sqrt{m})$ and $\End(A_{\Qbar}) \otimes_\Z \Q \simeq B$ will give rise to a rational non-CM point on $X_0(D,1)/\la \w_m \ra,$ thus studying the rational points on Atkin-Lehner quotients can be seen as having applications to the Coleman conjecture.

\subsection{Defining Equations}\label{DefiningEquations}

In this subsection we give a brief survey of defining equations of Shimura curves in the literature. Prior to the work of Gonz\'{a}lez-Rotger \cite{GR04}, the only defining equations of Shimura curves that were known had genus $0$ or $1$ \cite{Elkies98,JL85,Kurihara79}. Kurihara subsequently conjectured equations of all Shimura curves of genus 2 and many of genus $3$ and $5$ but was unable to prove them \cite{Kurihara94}. There are exactly three Shimura curves $X_0(D,1)$ of genus 2 given by $D=26,38,58.$ In \cite{GR04} the authors proved that the equations for these three genus $2$ curves which Kurihara conjectured are correct. They also gave defining equations for all genus 2 Atkin-Lehner quotients $X_0(D,1)/\langle \w_m \rangle$ which are bielliptic, of which there are $10$. These are for the pairs
\begin{align*}(D,m) \in \{ &(91,91),(123,123),(141,141),(142,2),(155,155),\\ &(158,158),(254,254),(326,326),(446,446) \}.\end{align*} A complete set of rational points on the $10$ bielliptic curves $X_0(D,1)/\langle \w_m \rangle$ of genus $2$ is determined by \cite[Table 3]{BFGR06}. All have a non-empty set of rational points except $X_0(142,1)/\la \w_2 \ra$ which has no rational points. There are $29$ more curves $X_0(D,1)/\langle \w_m \rangle$ of genus 2 and the pairs $(D,m)$ are given by \cite[Lemma 4.1]{GR04}. There are $32$ values of $D$ for which $X_0(D,1)$ is bielliptic \cite[Theorem 7]{Rotger02}.

 There are $24$ values of $D$ for which $X_0(D,1)$ is hyperelliptic over $\Qbar$ and $21$ of these are hyperelliptic over $\Q$ \cite{Ogg83}.  These are organized by genus into the following table.
 
 \begin{table}[H]
 \begin{center}
\begin{tabular}{ c | c  }
genus & discriminant $D$   \\ 
2 & 26, 38, 58 \\
 3 & 35, 39, 51, 55, 57, 62, 69,  82, 94  \\
 4 &  74, 86 \\
 5 & 87, 93\\
 6 & 134\\
 7 & 95, 146, 111\\
 8 & \\
 9 & 119, 159, 194, 206 \\   
\end{tabular}
\caption{The Shimura curves $X_0(D,1)$ which are hyperelliptic over $\Qbar$. All are hyperelliptic over $\Q$ except $D=57,82$ and 93}
\end{center}
\end{table}

In \cite{Molina12} the author computes equations for $X_0(D,1)$ for $D=39$ and 55 which both have genus 3. Molina also gives equations for some Atkin-Lehner quotients of Shimura curves $X_0(D,1)/ \langle \w_m \rangle$ which are genus 2, these are for the pairs \cite[Table 2]{Molina12}
 $$(D,m) \in \{ (35,5), (51,17),(57,3),(65,13),(65,5),(69,23),(85,5),(85,85)\}.$$

 In a subsequent work \cite{GM16} the authors give a complete list of hyperelliptic genus $2$ curves $X_0(D,1)/\langle \w_m \rangle$ such that $X_0(D,1)$ is hyperelliptic of genus $3$ \cite[Table 4]{GM16} along with equations for these. 
 
 Guo-Yang give a complete list of Shimura curves $X_0(D,N)$ which are hyperelliptic over $\Qbar$  \cite[Appendix A]{GY17} and also include information about what the Atkin-Lehner involutions are for each curve.  There are $44$ such curves.  Note that every quotient curve $X_0(D,N)/\langle \w_m \rangle$ in \cite{Molina12, GM16} is the quotient of a curve in \cite{GY17} except for $X_0(65,1)/ \langle \w_{13} \rangle, \ X_0(65,1)/ \langle \w_{5} \rangle, \ X_0(85,1)/\la \w_{5} \ra$ and $X_0(85,1)/\la \w_{85} \ra$ which is because $X_0(65,1)$ and $X_0(85,1)$  are not geometrically hyperelliptic.

  \section{Computations}\label{Computations}
  
  In this section we shall give details about what data we computed and how.   We used the program MAGMA \cite{MAGMA}.  We have created a database of Shimura curves and their quotients which has a unique identifier (given by $D,N$ and $W$) for each curve and can be accessed online (see \S\ref{Appendix}), with the objects being MAGMA-readable.  The database contains all possible curve quotients of the form $X_0(D,N)/W$ such that $X_0(D,N)$ is geometrically hyperelliptic or is genus $0$ or $1$.  Equations for the genus $0$ and $1$ curves along with the number of rational points on the quotients already appear in the literature \cite[Theorem 5.2]{BFGR06}.  The geometrically hyperelliptic curves are from \cite[Appendix A]{GY17}.  The initial data we start with are defining equations for $X_0(D,N)$ and generators of $W(D,N)$. 
  
  There is a variety of features we would like to compute about the quotient curves $X_0(D,N)/W,$ first though we must have equations defining the curve.  Given a subgroup $W$ of $\Aut_\Q(X_0(D,N))$ (or any curve more generally) it is usually possible to compute equations defining the quotient curve using the function \texttt{CurveQuotient()}.  This function will also compute defining equations for the map $X_0(D,N) \rightarrow X_0(D,N)/W.$ There are instances when the function fails to compute a quotient curve,  often this is when the quotient has genus 0 and no rational points.  In this case we were able to take advantage of the fact that $X_0(D,N)$ is geometrically hyperelliptic,  so we can take the quotient by a maximal subgroup of $W$ not containing the hyperelliptic involution,  and then compose with the function \texttt{IsGeometricallyHyperelliptic()}.  This is possible because the quotient of a hyperelliptic curve is also hyperelliptic \cite[Proposition A.1.(vii)]{Poonen07}.

Once we have defining equations for $X_0(D,N)/W$,  it is possible to compute further information about the curve such as the set of rational points $(X_0(D,N)/W)(\Q)$ and which points are  CM as explained in \S\ref{CMpoints}.

  \subsection{Computing rational points}
  
  Given a curve $X/\Q$ of arbitrary genus $g$ we would like to know two things:
  \begin{enumerate}
  \item Is $X(\Q)$ non-empty?
  \item If $X(\Q)$ is non-empty then what is $X(\Q)?$
  \end{enumerate}
  
  Both of these questions are deep and difficult to answer in general.  For a genus 0 curve $X(\Q)$ is either infinite or empty,  and if a rational point on the curve exists it is straightforward to explicitly parameterize all of the rational points.  For a genus 1 curve (1) is the same as asking whether $X$ is an elliptic curve. This is not as straightforward as the case of genus 0 but in practice can often be done. Then to know the structure of $X(\Q)$ in the case it is non-empty is to ask about the structure of the Mordell-Weil group. Whilst there is no algorithm to do this which is guaranteed to terminate,  it is a well-studied problem and we can call \texttt{MordellWeilGroup()} in MAGMA which is likely to succeed when the coefficients are small.      
  
If the genus of $X$ is greater than or equal to 2 then it is a theorem of Faltings \cite{Faltings83} that $X(\Q)$ is finite (also known as Mordell's conjecture).  The proof of this fact is not constructive and in practice there are many ways to attempt to answer questions (1) and (2), although no algorithm is guaranteed to work.  Recently, making $X(\Q)$ explicitly computable for certain higher genus curves has been a very active area of study,  with many breakthroughs (for example \cite{BDMTV19}).

The following is a brief survey of the techniques we used to answer questions (1) and (2) for the Atkin-Lehner quotients $X_0(D,N)/W.$

\noindent \textbf{Not everywhere locally solvable.} If a curve does not have any known rational points, one strategy to try is to check whether it has points everywhere locally. If it does not, then the curve has no $\Q$-rational points.  If the curve has genus $0$ or it has an equation of the form $y^2=f(x)$,  then one can easily compute the set $X(\A_f)$,  where $\A_f$ is the ring of finite adeles. For genus 0 and 1 curves we can use the Magma function \texttt{IsLocallySolvable()}, while higher genus curves are addressed using \texttt{HasPointsEverywhereLocally()}.

 \noindent \textbf{Two Cover Descent.} The Two Cover Descent method is particularly useful in showing that the set of $\Q$-rational points of a curve is empty.  Let $X/\Q$ be a non-singular hyperelliptic curve of genus $g$, given by the affine model
 $$X : y^2 = f_nx^n + \cdots + f_0 = f(x),$$
 
 where $f(x)$ is square-free. If $n$ is odd, then $X$ has a rational Weierstrass point, so in the case where the curve has no known rational points we can assume that $n = 2g+2$. We then consider the algebra 
 $$A = k[x]/(f(x))$$
 where we write $\theta$ for the image of $x$ in $A$, which means $f(\theta) = 0$. We consider the set
 
 $$H_{\Q} = \{ \ \delta \in A^\times/{A^{\times 2}}\Q^\times \ | \ \text{Nm}_{A/\Q}(\delta) \in f_n {\Q^\times}^2  \ \}.$$
 
 If the set $H_{\Q}$ is empty, one can conclude that $X(\Q)$ is empty.  This follows from the fact that one can define a map $X(\Q) \rightarrow H_{\Q}$. The details can be found in \cite[\S2]{BS09}. The Magma function  \texttt{TwoCoverDescent()} computes the set $H_{\Q}$.
  
   \noindent \textbf{Chabauty--Coleman method.} A special case of the Mordell conjecture had been proved by Chabauty in 1941 \cite{Chabauty41} when $\text{rank}(\Jac(X)) =r < g$. The $p$-adic integration techniques developed in the 1980s by Coleman \cite{Coleman85} succeed in bounding - and are often able to explicitly compute - the set of $\Q$-rational points. The strategy is particularly effective for curves of small genus. 

Firstly,  we can compute an upper bound on the number of rational points. The following theorem due to Stoll is a refinement of the original Chabauty-Coleman bound:

  \begin{theorem}[{\cite[Corollary~6.7]{Stoll06}}] 
Let $X/\Q$ be a nice curve of genus $g \geq 2$, $r$ be the rank of its Jacobian over $\Q$ and $p$ be a prime of good reduction for $X$. If $r < g$ and $p > 2r+2$, then
\[
    \#X(\Q) \leq \#X(\F_p) + 2r.
\]
\end{theorem}
For further details see \cite{MP12}.

For curves of genus $2$ satisfying the Chabauty-Coleman condition $r < g$, there are two Magma functions, \texttt{Chabauty0()} and \texttt{Chabauty()}, which deal with $r=0$ and $r=1$, respectively.
For hyperelliptic curves with $g > 2$ and $r=0$ admitting an odd degree model one can use the algorithms developed by \cite{FFH21} to compute the $\Q$-rational points. Balakrishnan and Tuitman \cite{BT20} have implemented code to deal with genus $3$ curves satisfying the Chabauty-Coleman condition.  There are methods such as quadratic Chabauty that tackle the case $r=g$, which is significantly more difficult. It has been successfully used for curves with additional structure such as bielliptic genus $2$ curves \cite{BD18} and modular curves \cite{BDMTV19}.

\noindent  \textbf{Pullback of rational points.} Given a curve $X$ along with a group of automorphisms $W \leq \Aut(X)$ one can compute an explicit quotient map $\pi : X \longrightarrow X/W$.  If $P$ is a rational point on $X$ then $\pi(P)$ is rational on the quotient $X/W$. Thus, if we know $(X/W)(\Q)$ and this set is finite then we can pullback the points and $\pi^{-1}((X/W)(\Q))$ contains $X(\Q)$.  This can be particularly effective for Shimura curves since they have many quotient maps and the quotient curves will be of smaller genus,  often making it easier to compute their rational points.

\noindent \textbf{Genus $1$ Curves.} We can attempt to classify the rational points on a genus 1 curve; either by finding a rational point and computing the Mordell-Weil group or showing that the curve has no rational points.  Firstly,  if the curve has no points everywhere locally we can conclude that it has no rational points.  Else,  we reduce the size of the equations of the curve with \texttt{Reduce(Minimise(GenusOneModel()))} and search for rational points in a box.  If we find a point then it is possible to construct an isomorphism $\phi : X \longrightarrow E$ to an elliptic curve $E$ with a plane model.  Then we are able to compute the Mordell-Weil group of $E$ and in the case where the rank is 0 we can pull the rational torsion points back along $\phi$.  In our examples, either the curve had a rational point and we were able to compute the Mordell-Weil group of the Jacobian or was not locally solvable.

With a combination of the techniques described above we were able to provably compute the set of rational points on almost all of the Atkin-Lehner quotients in the database.  The only curves for which the above methods could not be applied were the genus $2$ curves
$$X_0(93,1)/\la \w_{93} \ra \ \ \ \ \ \ \text{and} \ \ \ \ \ \ X_0(10,19)/\la \w_{190} \ra.$$
Both have Jacobians with rank $2$ over $\Q$.  

It turns out that $X_0(93,1) /\la \w_{93} \ra$ is isomorphic to the modular curve $X_0(1,93)/\langle \w_3, \w_{31} \rangle,$ whose points have been computed in \cite[Table 1]{BGX21}.  The curve $X_0(10,19)/\langle \w_{190} \rangle$ is a genus $2$ bielliptic curve whose Jacobian is isogenous to a product of two rank $1$ elliptic curves. The computation of its set of $\Q$-rational points will appear in an upcoming work of Francesca Bianchi with the first named author \cite{BP22}.

\begin{remark}
We note some other exceptional isomorphisms between Shimura and modular curves of genus $2$: $X_0(91,1)/ \langle \w_{91} \ra \simeq X_0(1,91)/ \langle \w_{91} \rangle$ and $X_0(85,1)/\la \w_{85} \ra \simeq X_0(255)^*$.  These were kindly pointed out to us by Jennifer Balakrishnan and Nikola Adžaga respectively.  The rational points of $X_0(1,91)/ \langle \w_{91} \rangle$ are computed in \cite{BBBM21} and the rational points of $X_0(255)^*$ are computed in \cite{ACKP22}.  The isomorphism can be checked directly using the defining equations and the function \texttt{IsIsomorphic()}.
\end{remark}

\subsection{CM points}\label{CMpoints}

In this subsection we explain how in practice one can often use a combinatorial description to determine whether a rational point on the Atkin-Lehner quotient $X_0(D,N)/\la \w_m \ra$ is CM or not.

We carry forward the notation from \S\ref{Background} with the further assumption that $N$ is squarefree.  In addition,  $p$ will be an integer prime; $\Delta_R$ will be the discriminant of $R$; $f$ will be the conductor of $R$; $I(R)$ the group of fractional invertible ideals of $R$; $H_R$ the ring class field of $R$; $h_R$ the class number of $R$ and $\sigma_\fraka$ the element in $\Gal(H_R | K)$ under the Artin map.  Let $\big( \frac{K}{p} \big)$ be the Kronecker symbol and $\big( \frac{R}{p} \big) := \big( \frac{K}{p} \big)$ if $p \nmid f$ and 1 otherwise (sometimes called the Eichler symbol).

Below we collate information that can be used to count the number of CM points on quotient curves. 
We will make frequent use of the following quantities:
$$D(R) = \prod_{p |D, \ \big( \frac{R}{p} \big) = -1} p, \ \ \ N(R) = \prod_{p|N, \ \big( \frac{R}{p} \big) = 1} p, \ \ \ N^\ast(R) = \prod_{p|N,  \ p \nmid f, \big( \frac{R}{p} \big)=1} p.$$
Also, let $m_r = \gcd \left(m, \frac{DN}{D(R)N(R)}\right)$ and $\frakb$ be an invertible ideal of $R$ such that $\Nm(\frakb) = m_r$.

\begin{theorem}\label{CMtheorem}
\begin{enumerate}
\item The set $\CM(R)$ is nonempty if and only if $\frac{DN}{D(R)N^\ast(R)} | \Delta_R$ and in this case $\# \CM(R)  = 2^{ \# \{ p|D(R) N(R) \} } \cdot h(R).$
\item The set of fixed points of $X_0(D,N)$ under the involution $\w_m$ is
$$X_0(D,N)^{\la \w_m \ra} =
\begin{cases}
\CM(\Z[\sqrt{-1}]) \cup \CM(\Z[\sqrt{-2}]) &\text{if} \ m=2 \\
\CM(\Z[\sqrt{-m}]) \cup \CM\left(\Z\left[\frac{1+\sqrt{-m}}{2}\right]\right) &\text{if} \ m \equiv 3 \ \text{mod} \ 4 \\
\CM(\Z[\sqrt{-m}]) \ &\text{otherwise}.
\end{cases}
 $$
 \item Let $P\in \CM(R).$ \begin{enumerate}
 \item If $D(R)N^\ast(R) \neq 1$ then $\Q(P) = H_R$ and 
 $$\Q(\pi_m(P)) = 
 \begin{cases}
 H_R^{\sigma_b} &\text{if} \ m/m_r =1; \\
  H_R^{\sigma_{\frakb \fraka} \cdot c} \text{ where } \fraka \text{ satisfies } B \simeq \Big(\frac{-s,  D(R)N^\ast(R)\Nm(\fraka)}{\Q} \Big) &\text{if } m/m_r = D(R)N^\ast(R);  \\
  H_R &\text{otherwise.}
 \end{cases}$$
 \item If $D(R)N^\ast(R) = 1$ then $|H_R : \Q(P)| = 2$ and 
 $$\Q(\pi_m(P)) = 
 \begin{cases}
 H_R^{\la c \cdot \sigma_\fraka, \sigma_\frakb \ra} &\text{if} \ m/m_r = 1, \\
 H_R^{c \cdot \sigma_\fraka} &\text{otherwise},
 \end{cases}$$
 where $B \simeq \Big(\frac{-s,  \Nm(\fraka)}{\Q} \Big).$
 \end{enumerate}
\end{enumerate}
\end{theorem}

\begin{proof}
See \cite[\S5.1]{GR06} and the references therein.\footnote{There is a typo in the paper which misses the term $D(R)N^\ast(R)$ from the quaternion algebra in case 2 of (3)(a).}
\end{proof}

\begin{corollary}
For the Atkin-Lehner quotient $X_0(D,N)/\la \w_m \ra$, the number of $\Q$-rational CM points $\#\pi_m(CM(R))(\Q)$ is effectively computable.
\end{corollary}

Suppose $X_0(D,N)/\la \w_m \ra(\Q)$ is finite and non-empty.  Given a knowledge of the complete set of points on the Atkin-Lehner quotient, we wish to determine which of these points are $\CM$. The information we have about a rational point $Q \in X_0(D,N)/\la \w_m \ra(\Q)$ is the field of definition of the pullback, denoted $F_Q = \Q(\pi_m^{-1}(Q)).$ 
Fix a class field $H_R$.  There is a containment
$$\{ \ \pi_m(\CM(S))(\Q) \ | \ S \subset H_R \ \} \subseteq \{ \ Q \ | \ F_Q \subset H_R \ \}.$$
If these two sets are equal then we can conclude that every rational point in $\{ \ Q \ | \ F_Q \subset H_R \ \}$ is a CM point for one of the orders such that $S \subset H_R$. 

Note that if the two sets above are not equal it would not be possible to determine which points are CM using this method, however,  in our case the two sets were always equal.  Whilst we were able to find out which points are CM, it is not always possible to identify the order $R$ associated to each CM point.  For example, the two orders $\Z[\sqrt{-1}]$ and $\Z[2\sqrt{-1}]$ have the same ring class field $H_R = \Q(\sqrt{-1}),$ also see Example \ref{CM?}.

The index $|H_R : \Q(P)|$ is at most 2,  and for $\pi_m(P)$ to be rational we require that $\Q(P)$ is imaginary quadratic,  hence we consider CM orders of class number at most 2. We give the complete list of these in the table below.

\begin{center}
\begin{table}[H]
\setlength\tabcolsep{2.5pt}
\begin{tabularx}{\textwidth}{c c c c c c c c c c c c c c c c c c c c c c}
-$\Delta_K$ & 3 &  3 &  3 &  3 &  3 &  3 &  4 &  4 &  4 &  4 &  4 &  7 &  7 &  7 &  8 &  8 &  8 &  11 &  11 &  15 &  15   \\
 \cline{1-22}
  \\[-1em]
$f$ & 1 &  2 &  3 &  4 &  5 &  7 &  1 &  2 &  3 &  4 &  5 &  1 &  2 &  4 &  1 &  2 &  3 &  1 &  3 &  1 &  2     \\
 \cline{1-22}
  \\[-1em]
$h_R$ & 1 & 1 & 1 & 2 & 2 & 2 & 1 & 1 & 2 & 2 & 2 & 1 & 1 & 2 & 1 & 2 & 2 & 1 & 2 & 2 & 2    \\
 \\[-1em]
 \\[-1em]
 \cline{2-22}
  \\[-1em]
  \cline{2-22}
   \\[-1em]
 \\[-1em]
 \\[-1em]
&  19 & 20 &  24 &  35 & 40 &  43 &  51 &  52 &  67 &  88 &  91 &  115 &  123 &  148 &  163 &  187 &  232 &  235 &  267 &  403 &  427 \\
   \cline{2-22}
  \\[-1em]
&  1  & 1 &  1 &  1 &   1 &  1 &  1 &  1 &  1 &  1 &  1 &  1 &  1 &  1 &  1 &  1 &  1 &  1 &  1 &  1 &  1 \\
     \cline{2-22}
  \\[-1em]
& 1 & 2 & 2 & 2 &   2 & 1 & 2 & 2 & 1 & 2 & 2 & 2 & 2 & 2 & 1 & 2 & 2 & 2 & 2 & 2 & 2 
\end{tabularx}
\caption{CM orders of class number at most 2. }
\end{table}
\end{center}

 \section{Examples}\label{Examples}

 \begin{example}
 The following curves are all of the violations of the Hasse principle that we found: they have points everywhere locally but no $\Q$-rational points. One can easily check that they have points everywhere locally using the Magma function \texttt{HasPointsEverywhereLocally()}, while the fact that they do have not any $\Q$-rational points can be shown using the Magma function \texttt{TwoCoverDescent()}.  Note that defining equations would not have been needed to know the local points on these curves since in \cite[Theorem 3.1]{RSY05} the authors give a complete description of whether $X_0(D,1)/\la  \w_m \ra$ has points everywhere locally for any involution $\w_m,$ which extends a result of Clark which shows that $X_0(D,1)/\la  \w_D \ra(\A_\Q) \neq \emptyset$ \cite[Main Theorem 2]{Clark03}.  
 
  \begin{align*}
  X_0(87,1)/\la \w_3 \ra \ &: \
y^2 + (x^3 + x + 1)y = -x^6 - 7x^4 + 16x^3 + 5x^2 - 16x - 4  \\
X_0(93,1)/ \la \w_3 \ra \ &: \
 y^2 + (x^3 + x^2 + 1)y = -x^8 + 2x^7 - 10x^6 - 29x^5 - 21x^4 - 35x^3 - 28x^2 - 9x - 13 \\
X_0(119,1)/ \la \w_7 \ra \ &: \ 
 y^2 + (x^5 + x^4 + x^2 + x + 1)y = -3x^{10} + 3x^9 + 5x^8 - x^7 - 17x^6 - 2x^5 
 \\&+ 18x^4 + 6x^3 - 12x^2 - 9x + 4  \\
X_0(6,29)/\la \w_6 \ra  \ &: \ 
 y^2 + (x^2 + x + 1)y = -64x^6 - 192x^5 + 155x^4 + 630x^3 - 210x^2 - 557x + 219  \\
X_0(6,37)/\la \w_3 \ra \ &: \
 y^2 + (x^3 + 1)y = -7x^6 - 27x^5 + 104x^3 - 27x - 7 \\
X_0(39,2)/ \la \w_{78} \ra \ &: \
 y^2 + (x^4 + x^3 + x^2 + x + 1)y = -2x^8 + 11x^7 - 28x^6 + 20x^5 - 5x^4 + 20x^3 
 \\&- 28x^2 + 11x - 2 
  \end{align*}

To our knowledge, these are new examples of Hasse principle violations except $X_0(93,1)/\la \w_{3} \ra$ which can be found in \cite[Table 3]{RV14} (though not with defining equations).  Examples of Shimura curves and their Atkin-Lehner quotients which violate the Hasse principle have been exhibited and studied in the literature.  For example,  if the genus of $X_0(D,1)/\la \w_D \ra$ is at least 2 then there are infinitely many quadratic twists of $(X_0(D,1),\w_D)$ that violate the Hasse principle \cite[Main Theorem]{CS18}, whilst in \cite[Theorem 3]{Clark08}, a family is given which has a positive density of quadratic twists violating the Hasse principle.

Furthermore, Jordan showed that the genus $3$ curve $X_0(39,1)(\Q(\sqrt{-13}))$ violates the Hasse principle \cite[Example 6.4]{Jordan86} and it was later shown that it is accounted for by the Manin obstruction \cite{SS03}. For an imaginary quadratic field $K$, in \cite{SY04} the authors give simple conditions to ensure that $X_0(D,N)(K)$ violates the Hasse principle and they show that these are also accounted for by the Manin obstruction.  Clark then shows that for a fixed $D>546$ there are infinitely many imaginary quadratic fields $K$ such that $X_0(D,1)(K)$ violates the Hasse principle \cite{Clark09}.

\end{example}

\begin{example}
There are very few rational non-CM points on the curves $X_0(D,1)/\la \w_m \ra$ in Table 3. In fact there are only $6$,  with $2$ on each of the genus $1$ curves $X_0(26,1)/\la \w_{13} \ra$,  $X_0(35,1)/\la \w_{7} \ra$ and $X_0(38,1)/\la \w_{19} \ra.$ Let us apply the contents of \S\ref{CMpoints} to find the CM points of the curve $X_0(38,1)/ \la w_{19} \ra(\Q)$.

Let $D=38$, $N=1$ and $m=19$. Then there exists an optimal embedding of an order $R$ in Table 2 if and only if $\frac{38}{D(R)} | \Delta_R$. This is satisfied for many orders $R$ and overall there are $108$ CM points on $X_0(38,1)$ for all such $R$.   

We require that $D(R) \neq 1$ since $38\nmid \Delta_R$ for any $R$.  Recall that $m_r := \gcd \left(19,  \frac{38}{D(R)}\right)$.  Now we can determine which of the many CM points on the genus $3$ curve $X_0(38,1)$ project to a rational CM point on the genus $1$ curve $X_0(38,1)/\la \w_{19} \ra.$ By Theorem \ref{CMtheorem}(3),  we are in the situation of part (a) and so $\Q(P) = H_R$ which forces $h_R=1.$ The only possibility for the fixed field $\Q(\pi_m(P))$ to be rational is if complex conjugation is acting, i.e.  we are in the case $m/m_r = D(R).$

This leaves us with 
$$\frac{19}{\gcd(19,38/D(R))} = 
\begin{cases}
1 &\text{if } D(R) = 2, \\
19 &\text{if } D(R) = 19 \text{ or } 38
\end{cases}$$
and we conclude that $D(R) = 19$. So we must have $\left(\frac{K}{2}\right)=1$ or $2|f$ and $\left(\frac{K}{19}\right)=-1$. This leaves us with the order $R= \Z[\sqrt{-1}]$. Note that the $R$-CM points are not fixed under Atkin-Lehner, so there is exactly one CM point on the quotient.  

We are in a position to identify which point is CM. The curves have defining equations
\begin{align*} X_0(38,1) &: y^2 = -16x^6 - 59x^4 - 82x^2 - 19 \\ 
X_0(38,1)/\la w_{19} \ra &: -X^3 + XYZ + Y^2Z - 9XZ^2 + YZ^2 - 90Z^3 = 0
\end{align*}
and the projection is defined by equations
$$\pi_{19} : (x:y:z) \longmapsto (-4x^2z - 5z^3 :2x^2xz + 2y+ 2z^3 :z^3)$$
where we have taken the projective model of $X_0(38,1)$ by introducing the variable $z$.
The rational points are 
$$\{Q_1, Q_2, Q_3\} = \{ (0:-10:1), (0:1:0), (0:9:1) \}.$$
One can easily compute the pullbacks of these points:
$$\{\pi_{19}^{-1}(Q_1), \pi_{19}^{-1}(Q_2), \pi_{19}^{-1}(Q_3)\} = \left \{ \left( \pm \frac{ \sqrt{-5}}{4}: - \frac{19}{4} : 1\right),  ( 1 : \pm 4\sqrt{-1} : 0), \left ( \pm \frac{ \sqrt{-5}}{4}:  \frac{19}{4} : 1\right ) \right \}.$$

We conclude that $(0:1:0)$ is the only rational CM point on $X_0(38,1)/\la \w_{19} \ra$ and that the other two points are non-CM.
\end{example}
 
 \begin{example}\label{CM?}
 In this example we see a quotient for which it is not possible to say exactly what the CM field $K$ is for each point, even though we know all of them are CM. 
 
 Consider the Atkin-Lehner quotient
 $$X_0(35,1)/\la \w_5 \ra :   y^2 + (x^2 + 1)y = -7x^5 + x^4 + x^3 + x^2 + 5x - 4$$
 with rational points 
 $$\{ \ (-1,-1), (3/4,-25/32), \infty \ \}.$$
 The pullback along $\pi_5$ of each of these points are contained in the field $\Q(\sqrt{-7}).$ Now computing the number of rational CM points as detailed in the previous example, we find that there is exactly one rational point on the quotient for each of the orders $\Z\left[\frac{1+\sqrt{-7}}{2}\right], \Z[ \sqrt{-7}]$ and $\Z\left [\frac{1+\sqrt{-35}}{2}\right].$ However, the ring class field $\Q(\sqrt{-7})$ of the first two orders is contained in the ring class field of the third order, meaning we are unable to identify the correct CM field for each point.
 \end{example}
 
 \bibliography{Shim_references}
 
 \section{Appendix}\label{Appendix}
 
Here we detail some of the data that has been computed as discussed in the previous sections. Included are two tables containing different information about quotients of geometrically hyperelliptic Shimura curves.  Curves of smaller genus and the rational points on their quotients already appear in the literature (cf.  \S\ref{DefiningEquations}).

The first table contains all of the rational points  on all of the quotient curves $X_0(D,1)/\la w_m \ra$ of level 1 such that the set of points is finite and non-empty.  For each rational point it is indicated whether the point is CM and if so then the discriminant of the CM field is given where possible (cf. Example \ref{CM?}).  Defining equations for each curve are also included.  As in previous sections we use $D$ for the discriminant of the quaternion algebra,  $\w_m$ for the Atkin-Lehner involution,  $X_0(D,N)/\la \w_m \ra$ for the Atkin-Lehner quotient curve and $\Delta_K$ for the fundamental discriminant of the CM field $K$.  We use `yes' or `no' to indicate whether a rational point is CM or not. 

In the second table\footnote{Note that there was a typo in the defining equations of $X_0(82,1)$ in \cite{GY17} which has been corrected.} information about all Atkin-Lehner quotients of the aforementioned Shimura curves is provided.  Specifically,  for every discriminant $D$,  level $N$ and subgroup of Atkin-Lehner involutions $W$,  the genus $g$ and the number of rational points $n$ is given.  For further information about the quotient curves - including defining equations,  the set of rational points (if finite),  what method was used to compute the rational points and equations for the projection $X_0(D,N) \longrightarrow X_0(D,N)/W$ - the interested reader can access the database online\footnote{\url{https://github.com/ciaran-schembri}}.

\begin{remark}
The methods used to determine whether a rational point on $X_0(D,1)/\la \w_m \ra$ is CM can also be used when the level is not one.  
\end{remark}


 \begin{center}
 \begin{longtable}{p{1.8cm}<{\centering} p{1.8cm}<{\centering} p{2.5cm}<{\centering} p{2cm}<{\centering} p{2cm}<{\centering}  p{0.1cm}<{\centering}}
 \caption{Rational points on the curves $X_0(D,N)/\la \w_m \ra$}\\
 \cline{1-5}
 \\[-1em]
 D & $\w_m$ & $X_0(D,1)/\la \w_m \ra( \Q)$ & CM & $\Delta_K$ &  \\
 \\[-1em]
 \cline{1-5}
 \\[-1em]
 \cline{1-5}
26 & $\w_{2}$ & & & & \\ \cline{1-2}
& & $\infty$ & yes & $-52$ & \\ 
\multicolumn{6}{c}{} \\
\multicolumn{6}{l}{$ y^2 + xy + y = x^3 - x^2 - 213x - 1257 $} \\\multicolumn{6}{c}{} \\
 \cline{1-5}
26 & $\w_{13}$ & & & & \\ \cline{1-2}
& & $(4,-9)$ & yes & $-8$ & \\ 
& & $\infty$ & no &  & \\ 
& & $(4,4)$ & no &  & \\ 
\multicolumn{6}{c}{} \\
\multicolumn{6}{l}{$ y^2 + xy + y = x^3 - 5x - 8 $} \\\multicolumn{6}{c}{} \\
 \cline{1-5}
35 & $\w_{5}$ & & & & \\ \cline{1-2}
& & $(-1,-1)$ & yes & $-7$ or $-35$ & \\ 
& & $(3/4,-25/32)$ & yes & $-7$ or $-35$ & \\ 
& & $\infty$ & yes & $-7$ or $-35$ & \\ 
\multicolumn{6}{c}{} \\
\multicolumn{6}{l}{$ y^2 + (x^2 + 1)y = -7x^5 + x^4 + x^3 + x^2 + 5x - 4 $} \\\multicolumn{6}{c}{} \\
 \cline{1-5}
35 & $\w_{7}$ & & & & \\ \cline{1-2}
& & $(1,3)$ & yes & $-35$ & \\ 
& & $\infty$ & no &  & \\ 
& & $(1,-4)$ & no &  & \\ 
\multicolumn{6}{c}{} \\
\multicolumn{6}{l}{$ y^2 + y = x^3 + x^2 + 9x + 1 $} \\\multicolumn{6}{c}{} \\
 \cline{1-5}
38 & $\w_{2}$ & & & & \\ \cline{1-2}
& & $\infty$ & yes & $-19$ & \\ 
\multicolumn{6}{c}{} \\
\multicolumn{6}{l}{$ y^2 + xy + y = x^3 + x^2 - 70x - 279 $} \\\multicolumn{6}{c}{} \\
 \cline{1-5}
38 & $\w_{19}$ & & & & \\ \cline{1-2}
& & $(0,9)$ & yes & $-4$ & \\ 
& & $\infty$ & no &  & \\ 
& & $(0,-10)$ & no &  & \\ 
\multicolumn{6}{c}{} \\
\multicolumn{6}{l}{$ y^2 + xy + y = x^3 + 9x + 90 $} \\\multicolumn{6}{c}{} \\
 \cline{1-5}
51 & $\w_{3}$ & & & & \\ \cline{1-2}
& & $\infty$ & yes & $-51$ & \\ 
\multicolumn{6}{c}{} \\
\multicolumn{6}{l}{$ y^2 + y = x^3 + x^2 - 59x - 196 $} \\\multicolumn{6}{c}{} \\
 \cline{1-5}
51 & $\w_{17}$ & & & & \\ \cline{1-2}
& & $(0,0)$ & yes & $-3$ or $-51$ & \\ 
& & $(-4,34)$ & yes & $-3$ or $-51$ & \\ 
& & $(-3,15)$ & yes & $-3$ or $-51$ & \\ 
\multicolumn{6}{c}{} \\
\multicolumn{6}{l}{$ y^2 + (x^3 + x)y = -2x^6 - 15x^5 - 39x^4 - 21x^3 + 19x^2 - 3x $} \\\multicolumn{6}{c}{} \\
 \cline{1-5}
57 & $\w_{3}$ & & & & \\ \cline{1-2}
& & $(-1,1)$ & yes & $-19$ & \\ 
\multicolumn{6}{c}{} \\
\multicolumn{6}{l}{$ y^2 + (x^3 + x)y = -2x^6 + 9x^5 - 13x^4 - 6x^3 + 12x^2 - 6x - 1 $} \\\multicolumn{6}{c}{} \\
 \cline{1-5}
58 & $\w_{2}$ & & & & \\ \cline{1-2}
& & $\infty$ & yes & $-232$ & \\ 
\multicolumn{6}{c}{} \\
\multicolumn{6}{l}{$ y^2 + xy + y = x^3 + x^2 - 455x - 3951 $} \\\multicolumn{6}{c}{} \\
 \cline{1-5}
62 & $\w_{31}$ & & & & \\ \cline{1-2}
& & $(-1,0)$ & yes & $-4$ & \\ 
& & $\infty$ & yes & $-8$ & \\ 
\multicolumn{6}{c}{} \\
\multicolumn{6}{l}{$ y^2 + (x^2 + x)y = 2x^5 - x^4 - x^3 - 2x^2 + x + 5 $} \\\multicolumn{6}{c}{} \\
 \cline{1-5}
69 & $\w_{23}$ & & & & \\ \cline{1-2}
& & $(1/4,-5/32)$ & yes & $-3$ & \\ 
& & $\infty$ & yes & $-3$ & \\ 
\multicolumn{6}{c}{} \\
\multicolumn{6}{l}{$ y^2 + (x^2 + x)y = 3x^5 - 11x^4 + 13x^3 + 13x^2 - 1 $} \\\multicolumn{6}{c}{} \\
 \cline{1-5}
74 & $\w_{2}$ & & & & \\ \cline{1-2}
& & $(-5/3,-5/9)$ & yes & $-148$ & \\ 
\multicolumn{6}{c}{} \\
\multicolumn{6}{l}{$ y^2 + (x^2 + x)y = -6x^6 - 16x^5 - 8x^4 - 8x^3 - 7x^2 + 17x - 5 $} \\\multicolumn{6}{c}{} \\
 \cline{1-5}
74 & $\w_{37}$ & & & & \\ \cline{1-2}
& & $(0,0)$ & yes & $-8$ & \\ 
\multicolumn{6}{c}{} \\
\multicolumn{6}{l}{$ y^2 + (x^2 + x)y = -x^6 + x^5 + 2x^4 - 4x^3 - x^2 + 4x $} \\\multicolumn{6}{c}{} \\
 \cline{1-5}
86 & $\w_{2}$ & & & & \\ \cline{1-2}
& & $(-3,-3)$ & yes & $-43$ & \\ 
\multicolumn{6}{c}{} \\
\multicolumn{6}{l}{$ y^2 + (x^2 + x)y = -x^6 - 5x^5 - x^4 + 10x^3 - 24x^2 - 26x - 6 $} \\\multicolumn{6}{c}{} \\
 \cline{1-5}
86 & $\w_{43}$ & & & & \\ \cline{1-2}
& & $(2,-3)$ & yes & $-4$ & \\ 
\multicolumn{6}{c}{} \\
\multicolumn{6}{l}{$ y^2 + (x^2 + x)y = 2x^6 - 2x^5 - 11x^4 + 6x^3 + 20x^2 - 7x - 11 $} \\\multicolumn{6}{c}{} \\
 \cline{1-5}
87 & $\w_{29}$ & & & & \\ \cline{1-2}
& & $(-1/2,-9/16)$ & yes & $-3$ & \\ 
& & $(1/2,-11/16)$ & yes & $-3$ & \\ 
\multicolumn{6}{c}{} \\
\multicolumn{6}{l}{$ y^2 + (x^3 + x^2 + 1)y = 3x^8 + 18x^7 + 70x^6 + 93x^5 - 36x^4 - 81x^3  $} \\
\multicolumn{6}{l}{$ \ \ \ \ \ + 23x^2 + 14x - 5 $} \\
\multicolumn{6}{c}{} \\
 \cline{1-5}
93 & $\w_{93}$ & & & & \\ \cline{1-2}
& & $(4/3,-5)$ & yes & $-163$ & \\ 
& & $(4/3,8/27)$ & yes & $-163$ & \\ 
& & $(1,-3)$ & yes & $-4$ & \\ 
& & $(1,0)$ & yes & $-4$ & \\ 
& & $(-1,2)$ & yes & $-4$ & \\ 
& & $(-1,-1)$ & yes & $-4$ & \\ 
& & $(0,0)$ & yes & $-19$ & \\ 
& & $(0,-1)$ & yes & $-19$ & \\ 
& & $(-2,9)$ & yes & $-67$ & \\ 
& & $(-2,0)$ & yes & $-67$ & \\ 
& & $\infty$ & yes & $-7$ & \\ 
& & $(1/2,-1)$ & yes & $-7$ & \\ 
& & $\infty'$ & yes & $-7$ & \\ 
& & $(1/2,-5/8)$ & yes & $-7$ & \\ 
\multicolumn{6}{c}{} \\
\multicolumn{6}{l}{$ y^2 + (x^3 + x + 1)y = x^3 + x^2 - 2x $} \\\multicolumn{6}{c}{} \\
 \cline{1-5}
94 & $\w_{47}$ & & & & \\ \cline{1-2}
& & $(-2,-1)$ & yes & $-4$ & \\ 
& & $\infty$ & yes & $-8$ & \\ 
\multicolumn{6}{c}{} \\
\multicolumn{6}{l}{$ y^2 + (x^2 + x)y = -2x^5 - 3x^4 - x^3 - 2x^2 + 7x - 3 $} \\\multicolumn{6}{c}{} \\
 \cline{1-5}
119 & $\w_{17}$ & & & & \\ \cline{1-2}
& & $(-1,-1)$ & yes & $-7$ & \\ 
& & $(3,-419)$ & yes & $-7$ & \\ 
\multicolumn{6}{c}{} \\
\multicolumn{6}{l}{$ y^2 + (x^6 + x^4 + x^3 + 1)y = -3x^{12} + 2x^{11} + 20x^{10} - 33x^8 + 31x^7  $} \\
\multicolumn{6}{l}{$  \ \ \ \ \ + 64x^6 - 47x^5 - 55x^4 + 50x^3 + 21x^2 - 34x - 13 $} \\
\multicolumn{6}{c}{} \\
 \cline{1-5}
134 & $\w_{2}$ & & & & \\ \cline{1-2}
& & $(-1/2,-1/16)$ & yes & $-67$ & \\ 
\multicolumn{6}{c}{} \\
\multicolumn{6}{l}{$ y^2 + (x^3 + x^2)y = -19x^8 + 53x^7 + 21x^6 - 143x^5 + 5x^4 + 143x^3  $} \\
\multicolumn{6}{l}{$\ \ \ \ \  + 5x^2 - 56x - 16 $} \\
\multicolumn{6}{c}{} \\
 \cline{1-5}
134 & $\w_{67}$ & & & & \\ \cline{1-2}
& & $(1,-1)$ & yes & $-4$ & \\ 
\multicolumn{6}{c}{} \\
\multicolumn{6}{l}{$ y^2 + (x^3 + x^2)y = 2x^8 - 2x^7 - 13x^6 + 14x^5 + 20x^4 - 35x^3 $} \\
\multicolumn{6}{l}{$ \ \ \ \ \  + 13x^2 + 32x - 32 $} \\
\multicolumn{6}{c}{} \\
 \cline{1-5}
159 & $\w_{53}$ & & & & \\ \cline{1-2}
& & $\infty$ & yes & $-3$ & \\ 
& & $(1/4,-1109/2048)$ & yes & $-3$ & \\ 
\multicolumn{6}{c}{} \\
\multicolumn{6}{l}{$ y^2 + (x^5 + x^4 + x^3 + x^2 + 1)y = 3x^{11} + 25x^{10} - 46x^9 - 483x^8 $} \\
\multicolumn{6}{l}{$\ \ \ \ \ + 571x^7 + 1063x^6 - 286x^5 - 1361x^4 + 1221x^3 - 437x^2 + 73x - 5   $}\\
\multicolumn{6}{c}{} \\
 \cline{1-5}
206 & $\w_{103}$ & & & & \\ \cline{1-2}
& & $(0,0)$ & yes & $-4$ & \\ 
& & $\infty$ & yes & $-8$ & \\ 
\multicolumn{6}{c}{} \\
\multicolumn{6}{l}{$ y^2 + (x^5 + x^4 + x^3 + x^2)y = -2x^{11} + 3x^{10} + 10x^9 + 82x^8 + 54x^7  $} \\
\multicolumn{6}{l}{$\ \ \ \ \ - 184x^6 - 1662x^5 - 4971x^4 - 7210x^3 - 4556x^2 - 1024x $}\\
\multicolumn{6}{c}{} \\
 \cline{1-5}
\end{longtable}
\end{center}

\clearpage

\begin{multicols*}{3}
\TrickSupertabularIntoMulticols
\topcaption{$X_0(D,N)/W$}
\begin{supertabular}{p{0.3cm}<{\centering}  p{0.3cm}<{\centering}  p{0.3cm}<{\centering} p{2.1cm}<{\centering}  p{0.3cm}<{\centering}}
\hline
 \\[-1em]
$D$ & $N$ & $g$ & $W$ gens. & $n$ \\
 \\[-1em]
 \hline
 \\[-1em]
\hline 
 26 & 1 & 2 &  & 0 \\ 
  &  & 1 & $\w_{2}$ & 1 \\ 
  &  & 1 & $\w_{13}$ & 3 \\ 
  &  & 0 & $\w_{26}$ & $\infty$ \\ 
  &  & 0 & $\w_{2},\w_{13}$ & $\infty$ \\ 
\hline \hline
 35 & 1 & 3 &  & 0 \\ 
  &  & 2 & $\w_{5}$ & 3 \\ 
  &  & 1 & $\w_{7}$ & 3 \\ 
  &  & 0 & $\w_{35}$ & $\infty$ \\ 
  &  & 0 & $\w_{5},\w_{7}$ & $\infty$ \\ 
\hline \hline
 38 & 1 & 2 &  & 0 \\ 
  &  & 1 & $\w_{2}$ & 1 \\ 
  &  & 1 & $\w_{19}$ & 3 \\ 
  &  & 0 & $\w_{38}$ & $\infty$ \\ 
  &  & 0 & $\w_{2},\w_{19}$ & $\infty$ \\ 
\hline \hline
 39 & 1 & 3 &  & 0 \\ 
  &  & 2 & $\w_{3}$ & 0 \\ 
  &  & 1 & $\w_{13}$ & 0 \\ 
  &  & 0 & $\w_{39}$ & $\infty$ \\ 
  &  & 0 & $\w_{3},\w_{13}$ & $\infty$ \\ 
\hline \hline
 51 & 1 & 3 &  & 0 \\ 
  &  & 1 & $\w_{3}$ & 1 \\ 
  &  & 2 & $\w_{17}$ & 3 \\ 
  &  & 0 & $\w_{51}$ & $\infty$ \\ 
  &  & 0 & $\w_{3},\w_{17}$ & $\infty$ \\ 
\hline \hline
 55 & 1 & 3 &  & 0 \\ 
  &  & 1 & $\w_{5}$ & 0 \\ 
  &  & 2 & $\w_{11}$ & 0 \\ 
  &  & 0 & $\w_{55}$ & $\infty$ \\ 
  &  & 0 & $\w_{5},\w_{11}$ & $\infty$ \\ 
\hline \hline
 57 & 1 & 3 &  & 0 \\ 
  &  & 2 & $\w_{3}$ & 1 \\ 
  &  & 0 & $\w_{19}$ & 0 \\ 
  &  & 1 & $\w_{57}$ & $\infty$ \\ 
  &  & 0 & $\w_{3},\w_{19}$ & $\infty$ \\ 
\hline \hline
 58 & 1 & 2 &  & 0 \\ 
  &  & 1 & $\w_{2}$ & 1 \\ 
  &  & 0 & $\w_{29}$ & $\infty$ \\ 
  &  & 1 & $\w_{58}$ & $\infty$ \\ 
  &  & 0 & $\w_{2},\w_{29}$ & $\infty$ \\ 
\hline \hline
 62 & 1 & 3 &  & 0 \\ 
  &  & 1 & $\w_{2}$ & 0 \\ 
  &  & 2 & $\w_{31}$ & 2 \\ 
  &  & 0 & $\w_{62}$ & $\infty$ \\ 
  &  & 0 & $\w_{2},\w_{31}$ & $\infty$ \\ 
\hline \hline
 69 & 1 & 3 &  & 0 \\ 
  &  & 1 & $\w_{3}$ & 0 \\ 
  &  & 2 & $\w_{23}$ & 2 \\ 
  &  & 0 & $\w_{69}$ & $\infty$ \\ 
  &  & 0 & $\w_{3},\w_{23}$ & $\infty$ \\ 
\hline \hline
 74 & 1 & 4 &  & 0 \\ 
  &  & 2 & $\w_{2}$ & 1 \\ 
  &  & 2 & $\w_{37}$ & 1 \\ 
  &  & 0 & $\w_{74}$ & $\infty$ \\ 
  &  & 0 & $\w_{2},\w_{37}$ & $\infty$ \\ 
\hline \hline
 82 & 1 & 3 &  & 0 \\ 
  &  & 2 & $\w_{2}$ & 0 \\ 
  &  & 0 & $\w_{41}$ & 0 \\ 
  &  & 1 & $\w_{82}$ & $\infty$ \\ 
  &  & 0 & $\w_{2},\w_{41}$ & $\infty$ \\ 
\hline \hline
 86 & 1 & 4 &  & 0 \\ 
  &  & 2 & $\w_{2}$ & 1 \\ 
  &  & 2 & $\w_{43}$ & 1 \\ 
  &  & 0 & $\w_{86}$ & $\infty$ \\ 
  &  & 0 & $\w_{2},\w_{43}$ & $\infty$ \\ 
\hline \hline
 87 & 1 & 5 &  & 0 \\ 
  &  & 2 & $\w_{3}$ & 0 \\ 
  &  & 3 & $\w_{29}$ & 2 \\ 
  &  & 0 & $\w_{87}$ & $\infty$ \\ 
  &  & 0 & $\w_{3},\w_{29}$ & $\infty$ \\ 
\hline \hline
 93 & 1 & 5 &  & 0 \\ 
  &  & 3 & $\w_{3}$ & 0 \\ 
  &  & 0 & $\w_{31}$ & 0 \\ 
  &  & 2 & $\w_{93}$ & 14 \\ 
  &  & 0 & $\w_{3},\w_{31}$ & $\infty$ \\ 
\hline \hline
 94 & 1 & 3 &  & 0 \\ 
  &  & 1 & $\w_{2}$ & 0 \\ 
  &  & 2 & $\w_{47}$ & 2 \\ 
  &  & 0 & $\w_{94}$ & $\infty$ \\ 
  &  & 0 & $\w_{2},\w_{47}$ & $\infty$ \\ 
\hline \hline
 95 & 1 & 7 &  & 0 \\ 
  &  & 3 & $\w_{5}$ & 0 \\ 
  &  & 4 & $\w_{19}$ & 0 \\ 
  &  & 0 & $\w_{95}$ & $\infty$ \\ 
  &  & 0 & $\w_{5},\w_{19}$ & $\infty$ \\ 
\hline \hline
 111 & 1 & 7 &  & 0 \\ 
  &  & 4 & $\w_{3}$ & 0 \\ 
  &  & 3 & $\w_{37}$ & 0 \\ 
  &  & 0 & $\w_{111}$ & $\infty$ \\ 
  &  & 0 & $\w_{3},\w_{37}$ & $\infty$ \\ 
\hline \hline
 119 & 1 & 9 &  & 0 \\ 
  &  & 4 & $\w_{7}$ & 0 \\ 
  &  & 5 & $\w_{17}$ & 2 \\ 
  &  & 0 & $\w_{119}$ & $\infty$ \\ 
  &  & 0 & $\w_{7},\w_{17}$ & $\infty$ \\ 
\hline \hline
 134 & 1 & 6 &  & 0 \\ 
  &  & 3 & $\w_{2}$ & 1 \\ 
  &  & 3 & $\w_{67}$ & 1 \\ 
  &  & 0 & $\w_{134}$ & $\infty$ \\ 
  &  & 0 & $\w_{2},\w_{67}$ & $\infty$ \\ 
\hline \hline
 146 & 1 & 7 &  & 0 \\ 
  &  & 4 & $\w_{2}$ & 0 \\ 
  &  & 3 & $\w_{73}$ & 0 \\ 
  &  & 0 & $\w_{146}$ & $\infty$ \\ 
  &  & 0 & $\w_{2},\w_{73}$ & $\infty$ \\ 
\hline \hline
 159 & 1 & 9 &  & 0 \\ 
  &  & 4 & $\w_{3}$ & 0 \\ 
  &  & 5 & $\w_{53}$ & 2 \\ 
  &  & 0 & $\w_{159}$ & $\infty$ \\ 
  &  & 0 & $\w_{3},\w_{53}$ & $\infty$ \\ 
\hline \hline
 194 & 1 & 9 &  & 0 \\ 
  &  & 5 & $\w_{2}$ & 0 \\ 
  &  & 4 & $\w_{97}$ & 0 \\ 
  &  & 0 & $\w_{194}$ & $\infty$ \\ 
  &  & 0 & $\w_{2},\w_{97}$ & $\infty$ \\ 
\hline \hline
 206 & 1 & 9 &  & 0 \\ 
  &  & 4 & $\w_{2}$ & 0 \\ 
  &  & 5 & $\w_{103}$ & 2 \\ 
  &  & 0 & $\w_{206}$ & $\infty$ \\ 
  &  & 0 & $\w_{2},\w_{103}$ & $\infty$ \\ 
\hline \hline
 6 & 11 & 3 &  & 0 \\ 
  &  & 2 & $\w_{2}$ & 0 \\ 
  &  & 2 & $\w_{3}$ & 0 \\ 
  &  & 1 & $\w_{6}$ & 0 \\ 
  &  & 2 & $\w_{11}$ & 0 \\ 
  &  & 1 & $\w_{22}$ & 0 \\ 
  &  & 1 & $\w_{33}$ & 0 \\ 
  &  & 0 & $\w_{66}$ & $\infty$ \\ 
  &  & 1 & $\w_{2},\w_{3}$ & 2 \\ 
  &  & 1 & $\w_{2},\w_{11}$ & 2 \\ 
  &  & 0 & $\w_{2},\w_{33}$ & $\infty$ \\ 
  &  & 1 & $\w_{3},\w_{11}$ & 2 \\ 
  &  & 0 & $\w_{3},\w_{22}$ & $\infty$ \\ 
  &  & 0 & $\w_{6},\w_{11}$ & $\infty$ \\ 
  &  & 0 & $\w_{6},\w_{22}$ & 0 \\ 
  &  & 0 & $\w_{2},\w_{3},\w_{11}$ & $\infty$ \\ 
\hline \hline
 6 & 17 & 3 &  & 0 \\ 
  &  & 1 & $\w_{2}$ & 0 \\ 
  &  & 2 & $\w_{3}$ & 0 \\ 
  &  & 2 & $\w_{6}$ & 0 \\ 
  &  & 2 & $\w_{17}$ & 0 \\ 
  &  & 0 & $\w_{34}$ & 0 \\ 
  &  & 1 & $\w_{51}$ & 2 \\ 
  &  & 1 & $\w_{102}$ & $\infty$ \\ 
  &  & 1 & $\w_{2},\w_{3}$ & 2 \\ 
  &  & 0 & $\w_{2},\w_{17}$ & $\infty$ \\ 
  &  & 0 & $\w_{2},\w_{51}$ & $\infty$ \\ 
  &  & 1 & $\w_{3},\w_{17}$ & 2 \\ 
  &  & 0 & $\w_{3},\w_{34}$ & $\infty$ \\ 
  &  & 1 & $\w_{6},\w_{17}$ & $\infty$ \\ 
  &  & 0 & $\w_{6},\w_{34}$ & $\infty$ \\ 
  &  & 0 & $\w_{2},\w_{3},\w_{17}$ & $\infty$ \\ 
\hline \hline
 6 & 19 & 3 &  & 0 \\ 
  &  & 2 & $\w_{2}$ & 0 \\ 
  &  & 1 & $\w_{3}$ & 0 \\ 
  &  & 2 & $\w_{6}$ & 2 \\ 
  &  & 1 & $\w_{19}$ & 0 \\ 
  &  & 2 & $\w_{38}$ & 2 \\ 
  &  & 1 & $\w_{57}$ & 0 \\ 
  &  & 0 & $\w_{114}$ & $\infty$ \\ 
  &  & 1 & $\w_{2},\w_{3}$ & 2 \\ 
  &  & 1 & $\w_{2},\w_{19}$ & 2 \\ 
  &  & 0 & $\w_{2},\w_{57}$ & $\infty$ \\ 
  &  & 0 & $\w_{3},\w_{19}$ & 0 \\ 
  &  & 0 & $\w_{3},\w_{38}$ & $\infty$ \\ 
  &  & 0 & $\w_{6},\w_{19}$ & $\infty$ \\ 
  &  & 1 & $\w_{6},\w_{38}$ & 2 \\ 
  &  & 0 & $\w_{2},\w_{3},\w_{19}$ & $\infty$ \\ 
\hline \hline
 6 & 29 & 5 &  & 0 \\ 
  &  & 2 & $\w_{2}$ & 0 \\ 
  &  & 3 & $\w_{3}$ & 0 \\ 
  &  & 2 & $\w_{6}$ & 0 \\ 
  &  & 3 & $\w_{29}$ & 0 \\ 
  &  & 2 & $\w_{58}$ & 0 \\ 
  &  & 3 & $\w_{87}$ & 2 \\ 
  &  & 0 & $\w_{174}$ & $\infty$ \\ 
  &  & 1 & $\w_{2},\w_{3}$ & 1 \\ 
  &  & 1 & $\w_{2},\w_{29}$ & 1 \\ 
  &  & 0 & $\w_{2},\w_{87}$ & $\infty$ \\ 
  &  & 2 & $\w_{3},\w_{29}$ & 3 \\ 
  &  & 0 & $\w_{3},\w_{58}$ & $\infty$ \\ 
  &  & 0 & $\w_{6},\w_{29}$ & $\infty$ \\ 
  &  & 1 & $\w_{6},\w_{58}$ & 1 \\ 
  &  & 0 & $\w_{2},\w_{3},\w_{29}$ & $\infty$ \\ 
\hline \hline
 6 & 31 & 5 &  & 0 \\ 
  &  & 3 & $\w_{2}$ & 0 \\ 
  &  & 2 & $\w_{3}$ & 0 \\ 
  &  & 2 & $\w_{6}$ & 0 \\ 
  &  & 3 & $\w_{31}$ & 0 \\ 
  &  & 3 & $\w_{62}$ & 2 \\ 
  &  & 2 & $\w_{93}$ & 0 \\ 
  &  & 0 & $\w_{186}$ & $\infty$ \\ 
  &  & 1 & $\w_{2},\w_{3}$ & 1 \\ 
  &  & 2 & $\w_{2},\w_{31}$ & 3 \\ 
  &  & 0 & $\w_{2},\w_{93}$ & $\infty$ \\ 
  &  & 1 & $\w_{3},\w_{31}$ & 1 \\ 
  &  & 0 & $\w_{3},\w_{62}$ & $\infty$ \\ 
  &  & 0 & $\w_{6},\w_{31}$ & $\infty$ \\ 
  &  & 1 & $\w_{6},\w_{62}$ & 1 \\ 
  &  & 0 & $\w_{2},\w_{3},\w_{31}$ & $\infty$ \\ 
\hline \hline
 6 & 37 & 5 &  & 0 \\ 
  &  & 2 & $\w_{2}$ & 0 \\ 
  &  & 2 & $\w_{3}$ & 0 \\ 
  &  & 3 & $\w_{6}$ & 0 \\ 
  &  & 2 & $\w_{37}$ & 0 \\ 
  &  & 3 & $\w_{74}$ & 2 \\ 
  &  & 3 & $\w_{111}$ & 2 \\ 
  &  & 0 & $\w_{222}$ & $\infty$ \\ 
  &  & 1 & $\w_{2},\w_{3}$ & 1 \\ 
  &  & 1 & $\w_{2},\w_{37}$ & 1 \\ 
  &  & 0 & $\w_{2},\w_{111}$ & $\infty$ \\ 
  &  & 1 & $\w_{3},\w_{37}$ & 1 \\ 
  &  & 0 & $\w_{3},\w_{74}$ & $\infty$ \\ 
  &  & 0 & $\w_{6},\w_{37}$ & $\infty$ \\ 
  &  & 2 & $\w_{6},\w_{74}$ & 3 \\ 
  &  & 0 & $\w_{2},\w_{3},\w_{37}$ & $\infty$ \\ 
\hline \hline
 10 & 11 & 5 &  & 0 \\ 
  &  & 2 & $\w_{2}$ & 0 \\ 
  &  & 3 & $\w_{5}$ & 0 \\ 
  &  & 2 & $\w_{10}$ & 0 \\ 
  &  & 3 & $\w_{11}$ & 0 \\ 
  &  & 2 & $\w_{22}$ & 0 \\ 
  &  & 3 & $\w_{55}$ & 2 \\ 
  &  & 0 & $\w_{110}$ & $\infty$ \\ 
  &  & 1 & $\w_{2},\w_{5}$ & 1 \\ 
  &  & 1 & $\w_{2},\w_{11}$ & 1 \\ 
  &  & 0 & $\w_{2},\w_{55}$ & $\infty$ \\ 
  &  & 2 & $\w_{5},\w_{11}$ & 3 \\ 
  &  & 0 & $\w_{5},\w_{22}$ & $\infty$ \\ 
  &  & 0 & $\w_{10},\w_{11}$ & $\infty$ \\ 
  &  & 1 & $\w_{10},\w_{22}$ & 3 \\ 
  &  & 0 & $\w_{2},\w_{5},\w_{11}$ & $\infty$ \\ 
\hline \hline
 10 & 13 & 3 &  & 0 \\ 
  &  & 2 & $\w_{2}$ & 0 \\ 
  &  & 2 & $\w_{5}$ & 0 \\ 
  &  & 1 & $\w_{10}$ & 0 \\ 
  &  & 1 & $\w_{13}$ & 0 \\ 
  &  & 2 & $\w_{26}$ & 0 \\ 
  &  & 0 & $\w_{65}$ & 0 \\ 
  &  & 1 & $\w_{130}$ & $\infty$ \\ 
  &  & 1 & $\w_{2},\w_{5}$ & 2 \\ 
  &  & 1 & $\w_{2},\w_{13}$ & 4 \\ 
  &  & 0 & $\w_{2},\w_{65}$ & $\infty$ \\ 
  &  & 0 & $\w_{5},\w_{13}$ & $\infty$ \\ 
  &  & 1 & $\w_{5},\w_{26}$ & $\infty$ \\ 
  &  & 0 & $\w_{10},\w_{13}$ & $\infty$ \\ 
  &  & 0 & $\w_{10},\w_{26}$ & $\infty$ \\ 
  &  & 0 & $\w_{2},\w_{5},\w_{13}$ & $\infty$ \\ 
\hline \hline
 10 & 19 & 5 &  & 0 \\ 
  &  & 2 & $\w_{2}$ & 0 \\ 
  &  & 3 & $\w_{5}$ & 0 \\ 
  &  & 2 & $\w_{10}$ & 0 \\ 
  &  & 3 & $\w_{19}$ & 0 \\ 
  &  & 0 & $\w_{38}$ & 0 \\ 
  &  & 3 & $\w_{95}$ & 2 \\ 
  &  & 2 & $\w_{190}$ & 14 \\ 
  &  & 1 & $\w_{2},\w_{5}$ & 1 \\ 
  &  & 0 & $\w_{2},\w_{19}$ & $\infty$ \\ 
  &  & 1 & $\w_{2},\w_{95}$ & $\infty$ \\ 
  &  & 2 & $\w_{5},\w_{19}$ & 3 \\ 
  &  & 0 & $\w_{5},\w_{38}$ & $\infty$ \\ 
  &  & 1 & $\w_{10},\w_{19}$ & $\infty$ \\ 
  &  & 0 & $\w_{10},\w_{38}$ & $\infty$ \\ 
  &  & 0 & $\w_{2},\w_{5},\w_{19}$ & $\infty$ \\ 
\hline \hline
 10 & 23 & 9 &  & 0 \\ 
  &  & 5 & $\w_{2}$ & 0 \\ 
  &  & 4 & $\w_{5}$ & 0 \\ 
  &  & 4 & $\w_{10}$ & 0 \\ 
  &  & 5 & $\w_{23}$ & 0 \\ 
  &  & 5 & $\w_{46}$ & 0 \\ 
  &  & 4 & $\w_{115}$ & 0 \\ 
  &  & 0 & $\w_{230}$ & $\infty$ \\ 
  &  & 2 & $\w_{2},\w_{5}$ & 1 \\ 
  &  & 3 & $\w_{2},\w_{23}$ & 3 \\ 
  &  & 0 & $\w_{2},\w_{115}$ & $\infty$ \\ 
  &  & 2 & $\w_{5},\w_{23}$ & 1 \\ 
  &  & 0 & $\w_{5},\w_{46}$ & $\infty$ \\ 
  &  & 0 & $\w_{10},\w_{23}$ & $\infty$ \\ 
  &  & 2 & $\w_{10},\w_{46}$ & 1 \\ 
  &  & 0 & $\w_{2},\w_{5},\w_{23}$ & $\infty$ \\ 
\hline \hline
 14 & 3 & 3 &  & 0 \\ 
  &  & 1 & $\w_{2}$ & 0 \\ 
  &  & 2 & $\w_{3}$ & 0 \\ 
  &  & 2 & $\w_{6}$ & 0 \\ 
  &  & 2 & $\w_{7}$ & 0 \\ 
  &  & 0 & $\w_{14}$ & 0 \\ 
  &  & 1 & $\w_{21}$ & 6 \\ 
  &  & 1 & $\w_{42}$ & 6 \\ 
  &  & 1 & $\w_{2},\w_{3}$ & 2 \\ 
  &  & 0 & $\w_{2},\w_{7}$ & $\infty$ \\ 
  &  & 0 & $\w_{2},\w_{21}$ & $\infty$ \\ 
  &  & 1 & $\w_{3},\w_{7}$ & 6 \\ 
  &  & 0 & $\w_{3},\w_{14}$ & $\infty$ \\ 
  &  & 1 & $\w_{6},\w_{7}$ & 6 \\ 
  &  & 0 & $\w_{6},\w_{14}$ & $\infty$ \\ 
  &  & 0 & $\w_{2},\w_{3},\w_{7}$ & $\infty$ \\ 
\hline \hline
 14 & 5 & 3 &  & 0 \\ 
  &  & 1 & $\w_{2}$ & 0 \\ 
  &  & 1 & $\w_{5}$ & 6 \\ 
  &  & 2 & $\w_{7}$ & 0 \\ 
  &  & 1 & $\w_{10}$ & 2 \\ 
  &  & 0 & $\w_{14}$ & $\infty$ \\ 
  &  & 2 & $\w_{35}$ & 0 \\ 
  &  & 2 & $\w_{70}$ & 0 \\ 
  &  & 0 & $\w_{2},\w_{5}$ & $\infty$ \\ 
  &  & 0 & $\w_{2},\w_{7}$ & $\infty$ \\ 
  &  & 1 & $\w_{2},\w_{35}$ & 2 \\ 
  &  & 1 & $\w_{5},\w_{7}$ & 6 \\ 
  &  & 0 & $\w_{5},\w_{14}$ & $\infty$ \\ 
  &  & 1 & $\w_{7},\w_{10}$ & 2 \\ 
  &  & 0 & $\w_{10},\w_{14}$ & $\infty$ \\ 
  &  & 0 & $\w_{2},\w_{5},\w_{7}$ & $\infty$ \\ 
\hline \hline
 15 & 2 & 3 &  & 0 \\ 
  &  & 2 & $\w_{2}$ & 0 \\ 
  &  & 1 & $\w_{3}$ & 2 \\ 
  &  & 2 & $\w_{5}$ & 4 \\ 
  &  & 2 & $\w_{6}$ & 0 \\ 
  &  & 1 & $\w_{10}$ & 0 \\ 
  &  & 0 & $\w_{15}$ & $\infty$ \\ 
  &  & 1 & $\w_{30}$ & 4 \\ 
  &  & 1 & $\w_{2},\w_{3}$ & 4 \\ 
  &  & 1 & $\w_{2},\w_{5}$ & 8 \\ 
  &  & 0 & $\w_{2},\w_{15}$ & $\infty$ \\ 
  &  & 0 & $\w_{3},\w_{5}$ & $\infty$ \\ 
  &  & 0 & $\w_{3},\w_{10}$ & $\infty$ \\ 
  &  & 1 & $\w_{5},\w_{6}$ & 8 \\ 
  &  & 0 & $\w_{6},\w_{10}$ & $\infty$ \\ 
  &  & 0 & $\w_{2},\w_{3},\w_{5}$ & $\infty$ \\ 
\hline \hline
 15 & 4 & 5 &  & 0 \\ 
  &  & 2 & $\w_{3}$ & 0 \\ 
  &  & 2 & $\w_{4}$ & 6 \\ 
  &  & 3 & $\w_{5}$ & 0 \\ 
  &  & 3 & $\w_{12}$ & 2 \\ 
  &  & 0 & $\w_{15}$ & 0 \\ 
  &  & 2 & $\w_{20}$ & 0 \\ 
  &  & 3 & $\w_{60}$ & 0 \\ 
  &  & 1 & $\w_{3},\w_{4}$ & 6 \\ 
  &  & 0 & $\w_{3},\w_{5}$ & $\infty$ \\ 
  &  & 1 & $\w_{3},\w_{20}$ & 2 \\ 
  &  & 1 & $\w_{4},\w_{5}$ & 4 \\ 
  &  & 0 & $\w_{4},\w_{15}$ & $\infty$ \\ 
  &  & 2 & $\w_{5},\w_{12}$ & 4 \\ 
  &  & 0 & $\w_{12},\w_{15}$ & $\infty$ \\ 
  &  & 0 & $\w_{3},\w_{4},\w_{5}$ & $\infty$ \\ 
\hline \hline
 21 & 2 & 3 &  & 0 \\ 
  &  & 1 & $\w_{2}$ & 0 \\ 
  &  & 2 & $\w_{3}$ & 2 \\ 
  &  & 2 & $\w_{6}$ & 2 \\ 
  &  & 0 & $\w_{7}$ & 0 \\ 
  &  & 2 & $\w_{14}$ & 0 \\ 
  &  & 1 & $\w_{21}$ & 8 \\ 
  &  & 1 & $\w_{42}$ & 4 \\ 
  &  & 1 & $\w_{2},\w_{3}$ & 4 \\ 
  &  & 0 & $\w_{2},\w_{7}$ & $\infty$ \\ 
  &  & 0 & $\w_{2},\w_{21}$ & $\infty$ \\ 
  &  & 0 & $\w_{3},\w_{7}$ & $\infty$ \\ 
  &  & 1 & $\w_{3},\w_{14}$ & 8 \\ 
  &  & 0 & $\w_{6},\w_{7}$ & $\infty$ \\ 
  &  & 1 & $\w_{6},\w_{14}$ & 8 \\ 
  &  & 0 & $\w_{2},\w_{3},\w_{7}$ & $\infty$ \\ 
\hline \hline
 22 & 3 & 3 &  & 0 \\ 
  &  & 2 & $\w_{2}$ & 0 \\ 
  &  & 1 & $\w_{3}$ & 0 \\ 
  &  & 2 & $\w_{6}$ & 2 \\ 
  &  & 1 & $\w_{11}$ & 0 \\ 
  &  & 2 & $\w_{22}$ & 6 \\ 
  &  & 1 & $\w_{33}$ & 0 \\ 
  &  & 0 & $\w_{66}$ & $\infty$ \\ 
  &  & 1 & $\w_{2},\w_{3}$ & 2 \\ 
  &  & 1 & $\w_{2},\w_{11}$ & 10 \\ 
  &  & 0 & $\w_{2},\w_{33}$ & $\infty$ \\ 
  &  & 0 & $\w_{3},\w_{11}$ & 0 \\ 
  &  & 0 & $\w_{3},\w_{22}$ & $\infty$ \\ 
  &  & 0 & $\w_{6},\w_{11}$ & $\infty$ \\ 
  &  & 1 & $\w_{6},\w_{22}$ & 6 \\ 
  &  & 0 & $\w_{2},\w_{3},\w_{11}$ & $\infty$ \\ 
\hline \hline
 22 & 5 & 5 &  & 0 \\ 
  &  & 2 & $\w_{2}$ & 0 \\ 
  &  & 2 & $\w_{5}$ & 0 \\ 
  &  & 3 & $\w_{10}$ & 2 \\ 
  &  & 2 & $\w_{11}$ & 0 \\ 
  &  & 3 & $\w_{22}$ & 0 \\ 
  &  & 3 & $\w_{55}$ & 2 \\ 
  &  & 0 & $\w_{110}$ & $\infty$ \\ 
  &  & 1 & $\w_{2},\w_{5}$ & 1 \\ 
  &  & 1 & $\w_{2},\w_{11}$ & 5 \\ 
  &  & 0 & $\w_{2},\w_{55}$ & $\infty$ \\ 
  &  & 1 & $\w_{5},\w_{11}$ & 1 \\ 
  &  & 0 & $\w_{5},\w_{22}$ & $\infty$ \\ 
  &  & 0 & $\w_{10},\w_{11}$ & $\infty$ \\ 
  &  & 2 & $\w_{10},\w_{22}$ & 3 \\ 
  &  & 0 & $\w_{2},\w_{5},\w_{11}$ & $\infty$ \\ 
\hline \hline
 26 & 3 & 5 &  & 0 \\ 
  &  & 2 & $\w_{2}$ & 0 \\ 
  &  & 3 & $\w_{3}$ & 0 \\ 
  &  & 2 & $\w_{6}$ & 0 \\ 
  &  & 3 & $\w_{13}$ & 0 \\ 
  &  & 0 & $\w_{26}$ & 0 \\ 
  &  & 3 & $\w_{39}$ & 2 \\ 
  &  & 2 & $\w_{78}$ & 6 \\ 
  &  & 1 & $\w_{2},\w_{3}$ & 1 \\ 
  &  & 0 & $\w_{2},\w_{13}$ & $\infty$ \\ 
  &  & 1 & $\w_{2},\w_{39}$ & 7 \\ 
  &  & 2 & $\w_{3},\w_{13}$ & 5 \\ 
  &  & 0 & $\w_{3},\w_{26}$ & $\infty$ \\ 
  &  & 1 & $\w_{6},\w_{13}$ & 3 \\ 
  &  & 0 & $\w_{6},\w_{26}$ & $\infty$ \\ 
  &  & 0 & $\w_{2},\w_{3},\w_{13}$ & $\infty$ \\ 
\hline \hline
 39 & 2 & 7 &  & 0 \\ 
  &  & 4 & $\w_{2}$ & 0 \\ 
  &  & 4 & $\w_{3}$ & 0 \\ 
  &  & 3 & $\w_{6}$ & 4 \\ 
  &  & 3 & $\w_{13}$ & 0 \\ 
  &  & 4 & $\w_{26}$ & 0 \\ 
  &  & 0 & $\w_{39}$ & $\infty$ \\ 
  &  & 3 & $\w_{78}$ & 0 \\ 
  &  & 2 & $\w_{2},\w_{3}$ & 6 \\ 
  &  & 2 & $\w_{2},\w_{13}$ & 2 \\ 
  &  & 0 & $\w_{2},\w_{39}$ & $\infty$ \\ 
  &  & 0 & $\w_{3},\w_{13}$ & $\infty$ \\ 
  &  & 2 & $\w_{3},\w_{26}$ & 2 \\ 
  &  & 1 & $\w_{6},\w_{13}$ & 2 \\ 
  &  & 0 & $\w_{6},\w_{26}$ & $\infty$ \\ 
  &  & 0 & $\w_{2},\w_{3},\w_{13}$ & $\infty$ \\
  \hline 
\end{supertabular}
 \end{multicols*}

\end{document}